\RequirePackage{rotating}
\documentclass[12pt]{article}

\usepackage{tikz}
\usetikzlibrary{decorations.markings}

\usepackage[margin=1in]{geometry}

\usepackage{amsmath}
\usepackage{amssymb}
\usepackage{url}
\usepackage{MnSymbol}

\usepackage{rotating}
\usepackage{mathdots}

\usepackage[colorlinks,pdfborder={0 0 0},linkcolor={red!80!black},citecolor={green!50!black},urlcolor={blue!80!black}]{hyperref}
\usepackage{amsthm}
\usepackage[capitalize]{cleveref}

\usepackage[all,cmtip]{xy}

\newtheorem{theorem}{Theorem}[section]
\newtheorem{lemma}[theorem]{Lemma}

\newcounter{claims}[theorem]
\newtheorem{claim}[claims]{Claim}

\theoremstyle{definition}
\newtheorem{definition}[theorem]{Definition}

\newtheorem*{question*}{Question}

\theoremstyle{remark}

\newcommand{\mc}[1]{\mathcal{#1}}

\renewcommand\labelenumi{(\arabic{enumi})}
\renewcommand\theenumi\labelenumi

\begin{document}

\title{Degrees of Categoricity Above Limit Ordinals}

\author{Barbara F. Csima\thanks{Partially supported by Canadian NSERC Discovery Grant 312501.}, Michael Deveau \thanks{Partially supported by Canadian NSERC Postgraduate
Scholarship PGSD1-234567-2017.}, \\Matthew Harrison-Trainor\thanks{Supported by an NSERC Banting Fellowship.}, Mohammad Assem Mahmoud}

\maketitle

\abstract{A computable structure $\mc{A}$ has degree of categoricity \textbf{d} if \textbf{d} is exactly the degree of difficulty of computing isomorphisms between isomorphic computable copies of $\mc{A}$. Fokina, Kalimullin, and Miller showed that every degree d.c.e.\ in and above $\mathbf{0}^{(n)}$, for any $n < \omega$, and also the degree $\mathbf{0}^{(\omega)}$, are degrees of categoricity. Later, Csima, Franklin, and Shore showed that every degree $\mathbf{0}^{(\alpha)}$ for any computable ordinal $\alpha$, and every degree d.c.e.\ in and above $\mathbf{0}^{(\alpha)}$ for any successor ordinal $\alpha$, is a degree of categoricity. We show that every degree c.e.\ in and above $\mathbf{0}^{(\alpha)}$, for $\alpha$ a limit ordinal, is a degree of categoricity. We also show that every degree c.e.\ in and above $\mathbf{0}^{(\omega)}$ is the degree of categoricity of a prime model, making progress towards a question of Bazhenov and Marchuk.}

\section{Introduction}

Isomorphisms of the vector space $\mathbb{Q}^\mathbb{N}$ are easy to understand: given two presentations of this vector space, one chooses a basis of each and then bijectively identifies basis elements from one presentation with basis elements from the other, in any manner one likes. Finally, the map is extended linearly to produce the full isomorphism. This classical method is straightforward, but could a computer be given this task and create such an isomorphism? That is, if the presentations of $\mathbb{Q}^\mathbb{N}$ are computable -- so the computer knows simple facts about each copy -- could we write a program that would provide an effective isomorphism, outputting the corresponding element in the second copy when given some element in the first copy as input?

The answer is no. It has long been known that there are two computable presentations of $\mathbb{Q}^\mathbb{N}$, one with a computable basis, and the other without a computable basis. Certainly there can be no computable isomorphism between these presentations, as the image of the computable basis from the first presentation would be a computable basis in the second presentation. Indeed the problem of building an isomorphism between two computable presentations of $\mathbb{Q}^\mathbb{N}$ amounts to computing bases for each of them, since the process of matching bases and extending linearly is effective; a computer could easily be given an algorithm explaining how to do this. To give a basis of a computable presentation of $\mathbb{Q}^\mathbb{N}$ requires only the ability to answer single quantifier questions: Is the next potential basis member a linear combination of what we have already included? The Turing degree of the Halting set, denoted $\mathbf{0}'$, is able to answer such questions. We say that $\mathbb{Q}^\mathbb{N}$ is $\mathbf{0}'$-computably categorical because we can compute, using $\mathbf{0}'$, an isomorphism between any two computable presentations. More generally:

\begin{definition}
Let $\mathbf{d}$ be a Turing degree. A computable structure $\mc{A}$ is $\mathbf{d}$-computably categorical if, for every computable copy $\mc{B}$ of $\mc{A}$, there is a $\mathbf{d}$-computable isomorphism between $\mc{A}$ and $\mc{B}$.
\end{definition}

For many structures $\mc{A}$, there is a least degree $\mathbf{d}$ such that $\mc{A}$ is $\mathbf{d}$-computably categorical. In the case of $\mathbb{Q}^\mathbb{N}$, there are actually two computable presentations of $\mathbb{Q}^\mathbb{N}$ such that any isomorphism between the two computes $\mathbf{0}'$. So $\mathbf{0}'$ is exactly the difficulty of computing isomorphisms between copies of $\mathbb{Q}^\mathbb{N}$. This natural idea was formalized by Fokina, Kalimullin, and Miller \cite{FokinaKalimullinMiller10}.

\begin{definition}[Fokina, Kalimullin, and Miller \cite{FokinaKalimullinMiller10}]
A Turing degree $\mathbf{d}$ is said to be the degree of categoricity of a computable structure $\mc{A}$ if $\mathbf{d}$ is the least degree such that $\mc{A}$ is $\mathbf{d}$-computably categorical.
\end{definition}

The Turing degree $\mathbf{0'}$, the level of complexity of the Halting set, is arguably the most natural Turing degree after $\mathbf{0}$, the degree of the computable sets. The \emph{jump} operator in computability theory takes a set $A$ to the halting set relative to $A$, denoted $A'$, and gives rise to a corresponding operator on Turing degrees. Iterating the jump operator $n$ times on $\mathbf{0}$ is called the $n$-th jump of $\mathbf{0}$, denoted $\mathbf{0}^{(n)}$. Taking unions at limit ordinals, one can define $\mathbf{0}^{(\alpha)}$ for any computable ordinal $\alpha$. It turns out that these definitions are very robust, see Ash Knight \cite{AshKnight00}. These Turing degrees are very natural. Just as $\mathbf{0'}$ is able to answer single-quantifier questions, $\mathbf{0}^{(n)}$ can answer questions expressible with $n$ alternating quantifiers.

Fokina, Kalimullin, and Miller showed that every degree that can be realized as a difference of computably enumerable (d.c.e.)\ sets in and above $\mathbf{0}^{(n)}$, for any $n < \omega$, and also the degree $\mathbf{0}^{(\omega)}$, are degrees of categoricity. Later, Csima, Franklin, and Shore \cite{CsimaFranklinShore13} showed that every degree $\mathbf{0}^{(\alpha)}$ for any computable ordinal $\alpha$, and every degree d.c.e.\ in and above $\mathbf{0}^{(\alpha)}$ for any successor ordinal $\alpha$, is a degree of categoricity. Csima and Ng have announced a proof that every $\Delta^0_2$ degree is a degree of categoricity.

It is often the case when trying to prove some property $P(\alpha)$ for ordinals $\alpha$ that things get tricky at limit ordinals. Roughly speaking, if $\alpha$ is a successor ordinal and we know something must happen before $\alpha$, we can safely say it has happened by $\alpha - 1$. For $\alpha$ a limit ordinal, in such a situation there is no canonical choice of earlier ordinal to look at. This is why the methods of \cite{CsimaFranklinShore13} did not work above limit ordinals.

Our main result in this paper is:

\begin{theorem}\label{thm:cea0alpha-degreeofcat}
Let $\alpha$ be a computable limit ordinal and $\mathbf{d}$ a degree c.e.\ in and above $\mathbf{0}^{(\alpha)}$.
There is a computable structure with (strong) degree of categoricity $\mathbf{d}$.
\end{theorem}

\noindent This fills in a gap that was missing from \cite{CsimaFranklinShore13} above limit ordinals, making further progress towards Question 5.1 of that paper.

We have not yet explained what a \textit{strong} degree of categoricity is. For a long time, all of the known examples had the following property: If $\mc{A}$ had degree of categoricity $\mathbf{d}$, then there is a copy $\mc{B}$ of $\mc{A}$ such that every isomorphism between $\mc{A}$ and $\mc{B}$ computes $\mathbf{d}$. Thus, we can witness with just two computable copies the fact that $\mathbf{d}$ is the \textit{least} degree such that $\mc{A}$ is $\mathbf{d}$-computably categorical. In this case, we say that $\mc{A}$ has strong degree of categoricity $\mathbf{d}$. Recently, Bazhenov, Kalimullin, and Yamaleev \cite{BazhenovKalimullinYamaleev18} have shown that there is a c.e.\ degree $\mathbf{d}$ and a structure $\mc{A}$ with degree of categoricity $\mathbf{d}$, but $\mathbf{d}$ is not a strong degree of categoricity for $\mc{A}$. Csima and Stephenson \cite{CsimaStephenson} have shown that there is a structure of finite computable dimension that has a degree of categoricity but no strong degree of categoricity.

Recall that the \emph{theory} of a structure is the set of first order formulas true in the structure, and that models of the same theory need not be isomorphic. The \emph{type} of a tuple in a structure is the set of formulas (with the appropriate number of free variables) that the tuple satisfies in the structure. The types of a theory are the types that are realized by models of the theory. A type is called \emph{principal} if there is one formula from which the rest follow. A model of a theory is \emph{prime} if it elementarily embeds into all other models of the theory, and when everything is countable, this is the same as saying that the model only realizes principal types. In a sense, prime structures are the most basic or natural structures. Our second result gives progress towards a question of Bazhenov and Marchuk.

\begin{question*}[Bazhenov and Marchuk \cite{BazhenovMarchuk}]
What can be the degrees of categoricity of computable prime models?
\end{question*}

\noindent A computable prime model---in fact, as \cite{BazhenovMarchuk} shows, a computable homogeneous model---is always $\mathbf{0}^{(\omega+1)}$-categorical, as we can ask $\mathbf{0}^{(\omega + 1)}$ if two tuples satisfy the same type. Bazhenov and Marchuk construct a computable homogeneous model with degree of categoricty $\mathbf{0}^{(\omega + 1)}$. The complexity here is in the structure itself, rather than in the theory. To build a prime model with degree of categoricity $\mathbf{0}^{(\omega + 1)}$, the complexity must be in the theory: If $\mc{A}$ is a computable prime model of a theory $T$, then $\mc{A}$ is $T' \oplus \mathbf{0}^{(\omega)}$-categorical as $T'$ can decide whether a formula is complete, and $\mathbf{0}^{(\omega)}$ can decide whether a formula holds of a tuple in $\mc{A}$. We build a computable prime model with degree of categoricity $\mathbf{0}^{(\omega + 1)}$ (or any other degree c.e.\ in and above $\mathbf{0}^{(\omega)}$), showing that the bound cannot be lowered.

\begin{theorem}\label{thm:main-prime}
Let $\mathbf{d}$ be a degree c.e.\ in and above $\mathbf{0}^{(\omega)}$.
There is a computable prime model $\mc{A}$ with strong degree of categoricity $\mathbf{d}$.
\end{theorem}

Bazhenov and Marchuk stated in \cite{BazhenovMarchuk} that a careful examination of the structures constructed in \cite{FokinaKalimullinMiller10} shows they are prime models, so that all degrees d.c.e.\ in and above $\mathbf{0}^{(n)}$ for a finite $n$, as well as $\mathbf{0}^{(\omega)}$, are strong degrees of categoricity of prime models. Along the way to proving Theorem \ref{thm:main-prime} we verify in Lemma \ref{lem:moh} that the building blocks used by Csima, Franklin and Shore for their examples in \cite{CsimaFranklinShore13} are prime. This is enough to see that their structures realizing degrees of categoricity less than or equal to $\mathbf{0}^{(\omega)}$ are prime. However, the structure in \cite{CsimaFranklinShore13} with degree of categoricity $\mathbf{0}^{(\omega +1)}$ is not prime. With Theorem \ref{thm:main-prime}, we see that all known degrees of categoricity less than or equal to the $\mathbf{0}^{(\omega+1)}$ bound can be realized by a prime model.

\section{Categoricity Relative to Decidable Models}\label{sec:decidable}

As a warm-up to illustrate the methods used to prove these two theorems, we give a simple proof of a result of Goncharov \cite{Goncharov11} that for every c.e.\ degree $\mathbf{d}$, there is a decidable prime model with degree of categoricity $\mathbf{d}$ with respect to decidable copies. Recall that a structure is said to be decidable if its full elementary diagram is computable. In \cite{Goncharov11}, Goncharov made the following definitions:

\begin{definition}
Let  $\mathbf{d}$ be a Turing degree and $\mc{A}$ a decidable structure. Then $\mc{A}$ is \textit{$\mathbf{d}$-categorical with respect to decidable copies} if for every decidable copy $\mc{B}$ of $\mc{A}$, $\mathbf{d}$ computes an isomorphism between $\mc{A}$ and $\mc{B}$.
\end{definition}

\begin{definition}
Let  $\mathbf{d}$ be a Turing degree and $\mc{A}$ a decidable structure. Then $\mathbf{d}$ is the \textit{degree of categoricity of $\mc{A}$ with respect to decidable copies} if:
\begin{itemize}
	\item $\mc{A}$ is $\mathbf{d}$-categorical with respect to decidable copies, and
	\item whenever $\mc{A}$ is $\mathbf{c}$-categorical with respect to decidable copies, $\mathbf{c} \geq \mathbf{d}$.
\end{itemize}
\end{definition}

It is not hard to see that between any two decidable copies of a prime model, there is a $\mathbf{0}'$-computable isomorphism. Goncharov showed that any c.e.\ degree can be the degree of categoricity with respect to decidable copies of a prime model. We give a different proof, which we think is simpler, and which demonstrates some of the techniques that we will use later.

\begin{theorem}[Goncharov {\cite[Theorem 3]{Goncharov11}}]
Let $\mathbf{d}$ be a c.e.\ degree. Then there is a decidable prime model $\mc{M}$ which has strong degree of categoricity $\mathbf{d}$ with respect to decidable models.
\end{theorem}
\begin{proof}
Let $D \in \mathbf{d}$ be a c.e.\ set. We will construct the structures $\mc{M}$ and $\mc{N}$. They are the disjoint union of infinitely many structures $\mc{M}_n$ and $\mc{N}_n$, with $\mc{M}_n$ and $\mc{N}_n$ picked out by unary relations $R_n$. The $n$th sort will code whether $n \in D$. Fix $n$. $\mc{M}_n$ will have infinitely many elements $(a_i)_{i \in \omega}$. There will be infinitely many unary relations $(U_\ell)_{\ell \in \omega}$ defined on $\mc{M}_n$ so that:
\[ a_0 \in U_s \Longleftrightarrow n \in D_{\text{at $s$}}\]
where $n \in D_{\text{at $s$}}$ means that $n$ enters $D$ at exactly stage $s$, and
\[ a_i \notin U_s \text{ for $i > 0$ and all $s$}.\]
Similarly, $\mc{N}_n$ will have infinitely many elements $(b_i)_{i \in \omega}$ with the unary relations defined so that:
\[ b_i \in U_s \Longleftrightarrow \text{$i = s$ and $n \in D_{\text{at $s$}}$}.\]
It is easy to see that we can build computable copies of $\mc{M}$ and $\mc{N}$. These copies are in fact decidable.

\begin{claim}
$\mc{M}$ and $\mc{N}$ are decidable.
\end{claim}
\begin{proof}
Given a formula $\varphi(x_1,\ldots,x_n)$ with $k$ quantifiers and $a_{i_1},\ldots,a_{i_n} \in \mc{M}$, it is not hard to see that $\mc{M} \models \varphi(a_{i_1},\ldots,a_{i_n})$ if and only if the finite substructure $\mc{M}'$ of $\mc{M}$ whose domain consists of $a_1,\ldots,a_{k+n+1}$ and $a_{i_1},\ldots,a_{i_n}$ also has $\mc{M}' \models \varphi(a_{i_1},\ldots,a_{i_n})$. Thus $\mc{M}$ is decidable.

For $\mc{N}$, suppose we have a formula $\varphi(x_1,\ldots,x_n)$ with at most $k$ quantifiers and which uses only some subset of the relations $U_0,\ldots,U_k$. Let $b_{i_1},\ldots,b_{i_n}$ be elements of $\mc{N}$. Then $\mc{N} \models \varphi(b_{i_1},\ldots,b_{i_n})$ if and only if the finite substructure $\mc{N}'$ of $\mc{M}$ whose domain consists of $b_1,\ldots,b_{k+n+1}$ and $b_{i_1},\ldots,b_{i_n}$ also has $\mc{N}' \models \varphi(b_{i_1},\ldots,b_{i_n})$. Thus $\mc{N}$ is decidable.
\end{proof}

\begin{claim}
$\mc{M}$ and $\mc{N}$ are isomorphic.
\end{claim}
\begin{proof}
It suffices to show that for each $n$, $\mc{M}_n$ and $\mc{N}_n$ are isomorphic. If $n \notin D$, then $a_i \mapsto b_i$ induces an isomorphism between $\mc{M}_n$ and $\mc{N}_n$. If $n \in D_{\text{at $s$}}$, then the map\begin{align*}
a_0 &\mapsto b_{s} & \\
a_{i} &\mapsto b_{i-1} && \text{when $0< i \leq s$} \\
a_i &\mapsto b_i && \text{when $i > s$} \\
\end{align*}
is an isomorphism between $\mc{M}_n$ and $\mc{N}_n$.
\end{proof}

\begin{claim}
$\mc{M}$ and $\mc{N}$ are prime.
\end{claim}
\begin{proof}
It suffices to show that each $\mc{M}_n$ and $\mc{N}_n$ are prime, since these structures are determined inside $\mc{M}$ and $\mc{N}$ uniquely by the relation $R_n$. It is not hard to see that $\mc{M}_n$ and $\mc{N}_n$ are models of an $\aleph_0$-categorical theory, and hence are prime.
\end{proof}

\begin{claim}
Any isomorphism between $\mc{M}$ and $\mc{N}$ can compute $D$.
\end{claim}
\begin{proof}
Let $g$ be an isomorphism between $\mc{M}$ and $\mc{N}$. For each $n$, let $(a_i)_{i \in \omega}$ and $(b_i)_{i \in \omega}$ be the elements in the definition of $\mc{M}_n$ and $\mc{N}_n$. Let $s$ be such that $g(a_0) = b_s$. Then $n \in D$ if and only if $n \in D_s$.
\end{proof}

\begin{claim}
Given a computable copy $\widetilde{\mc{M}}$ of $\mc{M}$, $D$ can compute an isomorphism between $\mc{M}$ and $\widetilde{\mc{M}}$.
\end{claim}
\begin{proof}
For each $n$, let $\widetilde{\mc{M}}_n$ be the structure with domain $R_n$ in $\widetilde{\mc{M}}$. It suffices to compute an isomorphism $g$ between $\mc{M}_n$ and $\widetilde{\mc{M}}_n$ for each $n$. Let $(c_i)_{i \in \omega}$ be the elements of $\widetilde{\mc{M}}_n$. If $n \notin D$, no relation $U_j$ holds of any of the the elements $(a_i)_{i \in \omega}$ or $(c_i)_{i \in \omega}$. So $a_i \mapsto c_i$ is an isomorphism. On the other hand, if $n \in D$, then for some unique $s$, $a_0 \in U_s$. We can look for $c_k$ such that $c_k \in U_s$. Map $a_0$ to $c_k$; map each other $a_i$ to some other $c_i$.
\end{proof}

These claims complete the proof of the theorem.
\end{proof}

\section{Back-and-forth Trees}

Fix a path through $O$. We will identify computable ordinals with their notation on this path. We will always first fix a limit ordinal $\alpha$ and work below it. Recall that one can decide effectively whether a given computable ordinal is a limit ordinal or a successor ordinal. For each limit ordinal $\beta < \alpha$, fix a fundamental sequence for $\beta$, that is, an increasing sequence of successor ordinals whose limit is $\beta$.

Hirschfeldt and White defined, for each successor ordinal $\beta$, a pair of trees $\mc{A}_\beta$ and $\mc{E}_\beta$ which can be differentiated exactly by $\beta$ jumps. These trees are called \emph{back-and-forth trees}.

\begin{definition}[{\cite[Definition 3.1]{HirschfeldtWhite02}}]
Back-and-forth trees are defined recursively in $\beta$. We view these as structures in the language of graphs with the root node distinguished.

We take $\mathcal{A}_1$ to be the tree with just a root node and no children, and we take $\mathcal{E}_1$ to be the tree where the root node has infinitely many children, none of which have children. See \cref{fig:baf-tree-1}. We say that these trees have \emph{back-and-forth rank} 1.

Suppose $\beta$ is a successor ordinal. Define $\mathcal{A}_{\beta+1}$ as a root node with infinitely many children, each the root of a copy of $\mathcal{E}_\beta$, and define $\mathcal{E}_{\beta+1}$ as a root node with infinitely many children, each the root of a copy of $\mathcal{A}_\beta$, and also infinitely many other children, each the root of a copy of $\mathcal{E}_\beta$. See \cref{fig:baf-tree-succ}. These trees have back-and-forth rank $\beta+1$.

Now suppose $\beta$ is a non-zero limit ordinal, and let $\beta_0, \beta_1, \ldots$ be a fundamental sequence of successor ordinals for $\beta$, that is, a sequence of successor ordinals below $\beta$ with limit $\beta$. We first define a family of helper trees $\mathcal{L}_{\beta,k}$ where $k \in \omega \cup \{\infty\}$. Define $\mathcal{L}_{\beta,\infty}$ to consist of a root node whose children are root nodes of copies of $\mathcal{A}_{\beta_i}$, and such that each copy appears exactly once as a child. For $k \in \omega$, $\mathcal{L}_{\beta,k}$ has a root node whose children are root nodes of copies of $\mathcal{A}_{\beta_0}, \ldots, \mathcal{A}_{\beta_k}, \mathcal{E}_{\beta_{k+1}}, \mathcal{E}_{\beta_{k+2}}, \ldots$ where again each copy appears exactly once as a child. Such trees are shown in \cref{fig:baf-trees-L}. We say these trees have back-and-forth rank $\beta$.

We can now define $\mathcal{A}_{\beta+1}$ and $\mathcal{E}_{\beta+1}$ for the non-zero limit ordinal $\beta$. For $\mathcal{A}_{\beta+1}$, we have a root node with infinitely many children, each the root node of a copy of $\mathcal{L}_{\beta,k}$ such that for each $k \in \omega$, $\mathcal{L}_{\beta,k}$ appears infinitely many times. The definition of $\mathcal{E}_{\beta+1}$ is similar, except $k$ is drawn from $\omega \cup \{\infty\}$. See \cref{fig:baf-tree-limit}. These trees have back-and-forth rank $\beta+1$.
\end{definition}

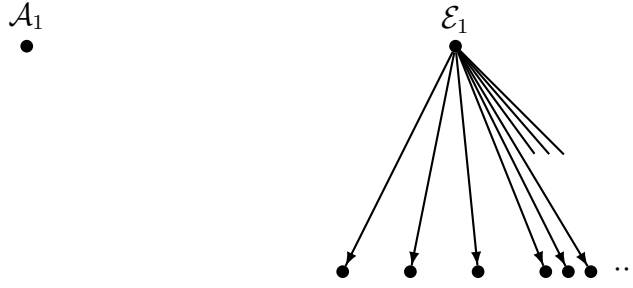
\begin{figure}
\centering
\begin{tikzpicture}[
  thick,
  every node/.style={circle,draw=black,fill=black,inner sep=0pt, minimum width=4pt},
  decoration={
    markings,
    mark=at position 1
    with {\arrow[black]{latex};}}
  ]

  \node (n0) at (-4.2,3.0) [label=above:$\mathcal{A}_1$] {};
  \node (n1) at (1.5,3.0) [label=above:$\mathcal{E}_1$] {};
  \node (n2) at (0.0,0.0) {};
  \node (n3) at (0.9,0.0) {};
  \node (n4) at (1.8,0.0) {};
  \node (n5) at (2.7,0.0) {};
  \node (n6) at (3.0,0.0) {};
  \node (n7) at (3.3,0.0) [label={[label distance=2mm]right:$\cdots$}] {};
  \node (n8) [draw=none,fill=none] at (2.6, 1.5) {};
  \node (n9) [draw=none,fill=none] at (2.8, 1.5) {};
  \node (n10) [draw=none,fill=none] at (3.0, 1.5) {};

  \foreach \from/\to in {n1/n2,n1/n3,n1/n4,n1/n5,n1/n6,n1/n7}
    \draw[postaction={decorate}] (\from) -- (\to);
  \foreach \from/\to in {n1/n8,n1/n9,n1/n10}
    \draw (\from) -- (\to);
\end{tikzpicture}
\caption{$\mathcal{A}_1$ and $\mathcal{E}_1$\label{fig:baf-tree-1}}
\end{figure}

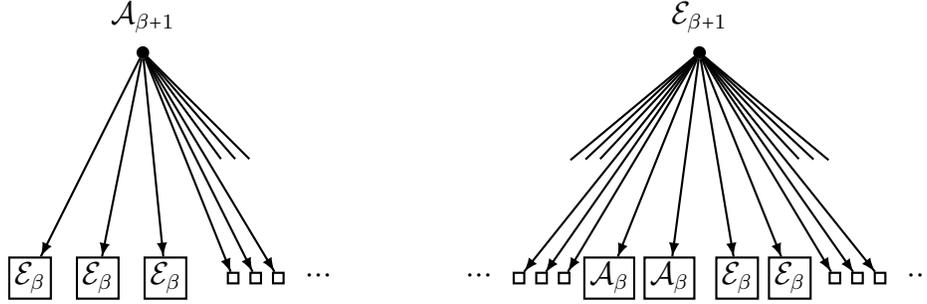
\begin{figure}
\centering
\begin{tikzpicture}[
  thick,
  every node/.style={draw=black,inner sep=2pt, minimum width=4pt, minimum height=4pt},
  decoration={
    markings,
    mark=at position 1
    with {\arrow[black]{latex};}}
  ]

  \node (n0) [circle,fill=black,inner sep=0pt] at (1.5,3.0) [label={[label distance=1mm]above:$\mathcal{A}_{\beta+1}$}] {};
  \node (n1) [circle,fill=black,inner sep=0pt] at (8.9,3.0) [label={[label distance=1mm]above:$\mathcal{E}_{\beta+1}$}] {};
  \node (n2) at (7.7,0.0) {$\mathcal{A}_{\beta}$};
  \node (n3) at (8.5,0.0) {$\mathcal{A}_{\beta}$};
  \node (n4) at (9.4,0.0) {$\mathcal{E}_{\beta}$};
  \node (n5) at (10.7,0.0) {};
  \node (n6) at (11.0,0.0) {};
  \node (n7) at (11.3,0.0) [label={[label distance=2mm]right:$\cdots$}] {};
  \node (n8) at (1.8,0.0) {$\mathcal{E}_{\beta}$};
  \node (n9) at (0.9,0.0) {$\mathcal{E}_{\beta}$};
  \node (n10) at (0.0,0.0) {$\mathcal{E}_{\beta}$};
  \node (n11) at (3.0,0.0) {};
  \node (n12) at (3.3,0.0) [label={[label distance=2mm]right:$\cdots$}] {};
  \node (n13) at (2.7,0.0) {};
  \node (n14) at (10.1,0.0) {$\mathcal{E}_{\beta}$};
  \node (n15) at (7.1,0.0) {};
  \node (n16) at (6.8,0.0) {};
  \node (n17) at (6.5,0.0) [label={[label distance=2mm]left:$\cdots$}] {};
  \node (n18) [draw=none,fill=none] at (10.3, 1.5) {};
  \node (n19) [draw=none,fill=none] at (10.5, 1.5) {};
  \node (n20) [draw=none,fill=none] at (10.7, 1.5) {};
  \node (n21) [draw=none,fill=none] at (2.6, 1.5) {};
  \node (n22) [draw=none,fill=none] at (2.8, 1.5) {};
  \node (n23) [draw=none,fill=none] at (3.0, 1.5) {};
  \node (n24) [draw=none,fill=none] at (7.1, 1.5) {};
  \node (n25) [draw=none,fill=none] at (7.3, 1.5) {};
  \node (n26) [draw=none,fill=none] at (7.5, 1.5) {};

  \foreach \from/\to in {n1/n2,n1/n3,n1/n4,n1/n5,n1/n6,n1/n7,n0/n10,n0/n9,n0/n8,n0/n11,n0/n12,n0/n13,n1/n14,n1/n17,n1/n16,n1/n15}
    \draw[postaction={decorate}] (\from) -- (\to);
  \foreach \from/\to in {n1/n18,n1/n19,n1/n20,n0/n21,n0/n22,n0/n23,n1/n24,n1/n25,n1/n26}
    \draw (\from) -- (\to);
\end{tikzpicture}
\caption{$\mathcal{A}_{\beta+1}$ and $\mathcal{E}_{\beta+1}$ when $\beta$ is a successor ordinal.\label{fig:baf-tree-succ}}
\end{figure}

\begin{figure}
\centering
\begin{tikzpicture}[
  thick,
  every node/.style={draw=black,inner sep=2pt, minimum width=4pt, minimum height=4pt},
  decoration={
    markings,
    mark=at position 1
    with {\arrow[black]{latex};}}
  ]

  \node (n0) [circle,fill=black,inner sep=0pt] at (1.5,3.0) [label={[label distance=1mm]above:$\mathcal{L}_{\beta,\infty}$}] {};
  \node (n1) [circle,fill=black,inner sep=0pt] at (8.9,3.0) [label={[label distance=1mm]above:$\mathcal{L}_{\beta,k}$}] {};
  \node (n2) at (5.5,0.0) {$\mathcal{A}_{\beta_0}$};
  \node (n3) at (7.6,0.0) {$\mathcal{A}_{\beta_k}$};
  \node (n4) at (8.7,0.0) {$\mathcal{E}_{\beta_{k+1}}$};
  \node (n5) at (10.7,0.0) {};
  \node (n6) at (11.0,0.0) {};
  \node (n7) at (11.3,0.0) [label={[label distance=2mm]right:$\cdots$}] {};
  \node (n8) at (1.8,0.0) {$\mathcal{A}_{\beta_2}$};
  \node (n9) at (0.9,0.0) {$\mathcal{A}_{\beta_1}$};
  \node (n10) at (0.0,0.0) {$\mathcal{A}_{\beta_0}$};
  \node (n11) at (3.0,0.0) {};
  \node (n12) at (3.3,0.0) [label={[label distance=2mm]right:$\cdots$}] {};
  \node (n13) at (2.7,0.0) {};
  \node (n14) at (9.9,0.0) {$\mathcal{E}_{\beta_{k+2}}$};
  \node (n15) at (6.8,0.0) {};
  \node (n16) at (6.5,0.0) {};
  \node (n17) at (6.2,0.0) {};
  \node (n18) [draw=none,fill=none] at (2.6, 1.5) {};
  \node (n19) [draw=none,fill=none] at (2.8, 1.5) {};
  \node (n20) [draw=none,fill=none] at (3.0, 1.5) {};
  \node (n21) [draw=none,fill=none] at (10.3, 1.5) {};
  \node (n22) [draw=none,fill=none] at (10.5, 1.5) {};
  \node (n23) [draw=none,fill=none] at (10.7, 1.5) {};

  \foreach \from/\to in {n1/n2,n1/n3,n1/n4,n1/n5,n1/n6,n1/n7,n0/n10,n0/n9,n0/n8,n0/n11,n0/n12,n0/n13,n1/n14,n1/n17,n1/n16,n1/n15}
    \draw[postaction={decorate}] (\from) -- (\to);
  \foreach \from/\to in {n0/n18,n0/n19,n0/n20,n1/n21,n1/n22,n1/n23}
    \draw (\from) -- (\to);
\end{tikzpicture}
\caption{Helper trees $\mathcal{L}_{\beta,\infty}$ and $\mathcal{L}_{\beta,k}$ for $k \in \omega$ for the non-zero limit ordinal $\beta$. \label{fig:baf-trees-L}}
\end{figure}
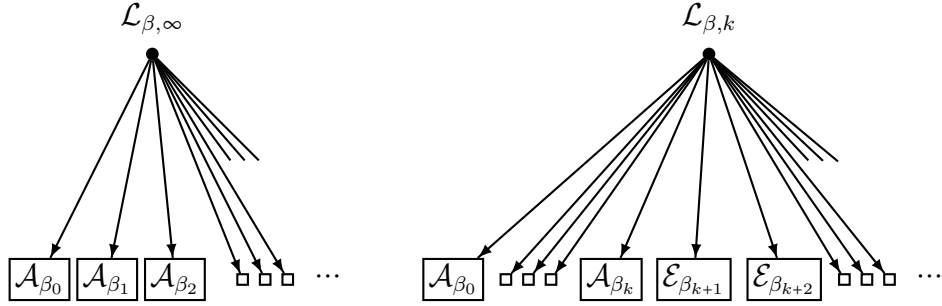

\begin{figure}
\centering
\begin{tikzpicture}[
  thick,
  every node/.style={draw=black,inner sep=2pt, minimum width=4pt, minimum height=4pt},
  decoration={
    markings,
    mark=at position 1
    with {\arrow[black]{latex};}}
  ]

  \node (n0) [circle,fill=black,inner sep=0pt] at (2.9,3.0) [label={[label distance=1mm]above:$\mathcal{A}_{\beta+1}$}] {};
  \node (n1) [circle,fill=black,inner sep=0pt] at (10.4,3.0) [label={[label distance=1mm]above:$\mathcal{E}_{\beta+1}$}] {};
  \node (n2) at (8.6,0.0) {};
  \node (n3) at (10.6,0.0) {$\mathcal{L}_{\beta,0}$};
  \node (n4) at (11.2,0.0) {};
  \node (n5) at (13.0,0.0) {};
  \node (n6) at (13.3,0.0) {};
  \node (n7) at (13.6,0.0) [label={[label distance=1mm]right:$\cdots$}] {};
  \node (n8) at (2.6,0.0) {$\mathcal{L}_{\beta,1}$};
  \node (n9) at (0.6,0.0) {};
  \node (n10) at (0.0,0.0) {$\mathcal{L}_{\beta,0}$};
  \node (n11) at (3.7,0.0) {};
  \node (n12) at (4.0,0.0) [label={[label distance=1mm]right:$\cdots$}] {};
  \node (n13) at (3.4,0.0) {};
  \node (n14) at (11.8,0.0) [label={[label distance=1mm]right:$\cdots$}] {};
  \node (n15) at (9.2,0.0) [label={[label distance=1mm]right:$\cdots$}] {};
  \node (n16) at (8.9,0.0) {};
  \node (n17) at (7.8,0.0) {$\mathcal{L}_{\beta,\infty}$};
  \node (n18) at (11.5,0.0) {};
  \node (n19) at (0.9,0.0) {};
  \node (n20) at (1.2,0.0) [label={[label distance=1mm]right:$\cdots$}] {};
  \node (n21) at (5.7,0.0) [label={[label distance=1mm]right:$\cdots$}] {};
  \node (n22) at (5.1,0.0) {};
  \node (n23) at (5.4,0.0) {};
  \node (n24) [draw=none,fill=none] at (2.2, 1.5) {};
  \node (n25) [draw=none,fill=none] at (2.4, 1.5) {};
  \node (n26) [draw=none,fill=none] at (2.6, 1.5) {};
  \node (n27) [draw=none,fill=none] at (3.6, 1.5) {};
  \node (n28) [draw=none,fill=none] at (3.8, 1.5) {};
  \node (n29) [draw=none,fill=none] at (4.5, 1.5) {};
  \node (n30) [draw=none,fill=none] at (4.7, 1.5) {};
  \node (n31) [draw=none,fill=none] at (4.9, 1.5) {};
  \node (n32) [draw=none,fill=none] at (10.0, 1.5) {};
  \node (n33) [draw=none,fill=none] at (10.2, 1.5) {};
  \node (n34) [draw=none,fill=none] at (10.4, 1.5) {};
  \node (n35) [draw=none,fill=none] at (11.3, 1.5) {};
  \node (n36) [draw=none,fill=none] at (11.5, 1.5) {};

  \node (n37) [draw=none,fill=none] at (12.2, 1.5) {};
  \node (n38) [draw=none,fill=none] at (12.4, 1.5) {};
  \node (n39) [draw=none,fill=none] at (12.6, 1.5) {};

  \foreach \from/\to in {n1/n2,n1/n3,n1/n4,n1/n5,n1/n6,n1/n7,n0/n10,n0/n9,n0/n8,n0/n11,n0/n12,n0/n13,n1/n14,n1/n17,n1/n16,n1/n15,n1/n18,n0/n19,n0/n20,n0/n21,n0/n22,n0/n23}
    \draw[postaction={decorate}] (\from) -- (\to);
  \foreach \from/\to in {n0/n24,n0/n25,n0/n26,n0/n27,n0/n28,n0/n29,n0/n30,n0/n31,n1/n32,n1/n33,n1/n34,n1/n35,n1/n36,n1/n37,n1/n38,n1/n39}
    \draw (\from) -- (\to);
\end{tikzpicture}
\caption{$\mathcal{A}_{\beta+1}$ and $\mathcal{E}_{\beta+1}$ for the non-zero limit ordinal $\beta$.\label{fig:baf-tree-limit}}
\end{figure}
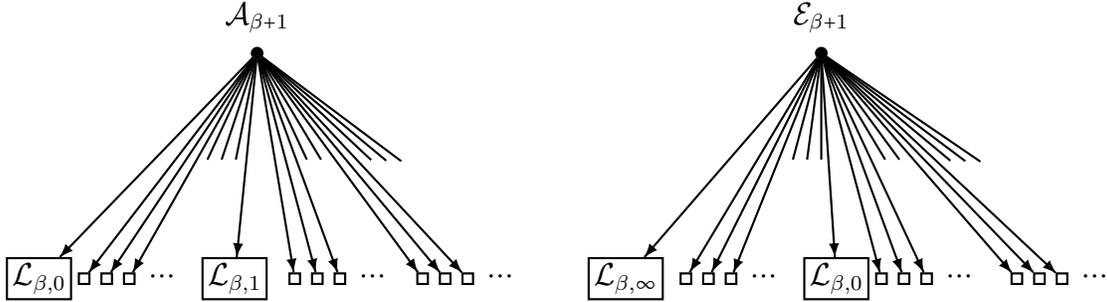

The next two lemmas piece together the facts that we will need about the back-and-forth trees, first for arbitrary $\beta$, and second some additional properties for finite $\beta$ in particular. These facts come from \cite{HirschfeldtWhite02} and \cite{CsimaFranklinShore13}.

\begin{lemma}\label{lem:michael}
Let $\alpha$ be a computable ordinal. For a successor ordinal $\beta < \alpha$, the structures $\mc{A}_\beta$ and $\mc{E}_\beta$ satisfy the following properties:
\begin{enumerate}
	\item Uniformly in $\beta$ and an index for a $\Sigma^0_\beta$ set $S$, there is a computable sequence of structures $\mc{C}_x$ such that
\[ x \in S \Longleftrightarrow \mc{C}_x \cong \mc{E}_\beta \quad \text{ and } \quad x \notin S \Longleftrightarrow \mc{C}_x \cong \mc{A}_\beta .\]
	\item Uniformly in $\beta$, there is a $\Sigma^0_\beta$ sentence $\varphi$ such that $\mc{E}_\beta \models \varphi$ and $\mc{A}_\beta \nmodels \varphi$.
	\item $\mc{A}_\beta$ and $\mc{E}_\beta$ are uniformly $\mathbf{0}^{(\beta)}$-categorical.
\end{enumerate}
\end{lemma}
\begin{proof}
For (1), take $(\mc{C}_x)_x$ to be the computable sequence of trees given by Proposition 3.2 in \cite{HirschfeldtWhite02}.

For (2), take $\varphi$ to be the sentence given by evaluating the formula guaranteed by Lemma 3.5 in \cite{HirschfeldtWhite02} for $\mc{B} = \mc{E}_\beta$ at its own root node. The complexity of $\varphi$ is the natural
complexity of $\mc{E}_\beta$, which is $\Sigma_\beta$. This lemma says that for any tree $\mc{T}$, $\mc{T} \models \varphi$ if and only if $\mc{T} \cong \mc{E}_\beta$.


Finally, for (3), we use a result from Csima, Franklin, and Shore \cite{CsimaFranklinShore13} about back-and-forth trees. We will consider $\mc{A}_\beta$; the case for $\mc{E}_\beta$ is identical. We have that $\mc{A}_\beta$ is a back-and-forth tree, and hence if $\mc{C}$ is a computable structure isomorphic to $\mc{A}_\beta$, then it is also a computable back-and-forth tree. Corollary 2.6 in \cite{CsimaFranklinShore13} allows $\emptyset^{(\gamma)}$ to uniformly compute an isomorphism between these two trees when the back-and-forth rank of the trees is at most $\gamma$. Since the rank of $\mc{A}_\beta$ is $\beta$ by construction, the isomorphism is uniformly computable in $\mathbf{0}^{(\beta)}$.
\end{proof}

Now for finite ordinals $\beta$ (and writing $n$ for $\beta$), we have some additional properties. We will state the lemma in full, including properties that were covered by the previous lemma. We think that these facts are well-known, but we do not know of a reference in print.

\begin{lemma}\label{lem:moh}
For $0 < n < \omega$, the structures $\mc{A}_n$ and $\mc{E}_n$ satisfy the properties:
\begin{enumerate}
	\item Uniformly in $n$ and an index for a $\Sigma^0_n$ set $S$, there is a computable sequence of structures $\mc{C}_x$ such that
	\[ x \in S \Longleftrightarrow \mc{C}_x \cong \mc{E}_n \quad \text{ and } \quad x \notin S \Longleftrightarrow \mc{C}_x \cong \mc{A}_n .\]
	\item For each $n$, there is an elementary first-order $\exists_n$ sentence $\varphi_n$, computable uniformly in $n$, such that $\mc{E}_n \models \varphi$ and $\mc{A}_n \nmodels \varphi$.
	\item $\mc{A}_n$ and $\mc{E}_n$ are prime.
	\item $\mc{A}_n$ and $\mc{E}_n$ are $\mathbf{0}^{(n)}$-categorical uniformly in $n$.
\end{enumerate}
\end{lemma}
\begin{proof}
(1) and (4) are the same as in the previous lemma. We show using induction on $n$ that these sequences satisfy (2) and (3) as well.   It is easy to see that $\mathcal{A}_1$ and $\mathcal{E}_1$ are prime models of their theories and that they are distinguishable (in the sense of (2) in the statement of the lemma) by the existential sentence $\varphi_1 := \exists x\exists y (x\neq y)$.
Assume now that $\mathcal{A}_n$ and $\mathcal{E}_n$ are prime and distinguishable by a first-order $\exists_{n}$ sentence $\varphi_n$ (in the sense that $\mc{E}_n \models \varphi_{n}$ but $\mc{A}_n \nmodels \varphi_{n}$).
We show that $\mathcal{A}_{n+1}$ and $\mathcal{E}_{n+1}$ are prime and distinguishable by a first-order $\exists_{n+1}$ sentence $\varphi_{n+1}$.

It is not hard to see that we can take $\varphi_{n+1}$ to be the sentence $\exists x \neg\varphi_n[\preceq x] \wedge (x \mbox{ is a child of the root node})$ where $x$ is a new variable not appearing in $\varphi_n$ and $\varphi_n[\preceq x]$ is the formula obtained from $\varphi_n$ by bounding every quantifier to the subtree below $x$. (Note that in a tree of rank $n$, if $z$ is a descendant of $x$, i.e. there is a path from $x$ to $z$, the length of the path is at most $n$, and so this is first-order definable and does not change the quantifier rank.) $\mc{E}_{n+1} \models \varphi_{n+1}$ but $\mc{A}_{n+1} \nmodels \varphi_{n+1}$.

It remains to show that $\mathcal{A}_{n+1}$ and $\mathcal{E}_{n+1}$ are prime. The same method will work for both structures. Let $\bar{a}$
be an arbitrary tuple in $\mathcal{E}_{n+1}$. We describe a formula that isolates the type of the tuple $\bar{a}$. Let $r_1,\ldots,r_k$ be the children of the root which are the roots of subtrees containing elements of $\bar{a}$; say that $\bar{a} = (\bar{a}_1,\ldots,\bar{a}_k)$ where $\bar{a}_i$ is in the subtree below $r_i$. (Note that we can re-order the tuples as we like, as if the type of some permutation of $\bar{a}$ is isolated, so is $\bar{a}$.) By the induction hypothesis, we know that the subtree with root $r_i$ is prime for every $i$. Hence for each $i = 1, \ldots,k$ there is a formula which isolates the type of the tuple $\bar{a}_i$ in the subtree with root $r_i$. There is also, for each $r_i$, a formula (either $\varphi_n$ or $\neg \varphi_n$) which distinguishes between whether the subtree below $r_i$ is isomorphic to $\mc{A}_n$ or $\mc{E}_n$. So we can isolate the type of $\bar{a}$ by saying that there are children $r_1,\ldots,r_k$ of the root such that $\bar{a}_i$ satisfies the formula, in the subtree below $r_i$, which isolates it, and by saying whether the subtree below each $r_i$ is isomorphic to $\mc{A}_n$ or $\mc{E}_n$.
\end{proof}

Fokina, Kalimullin, and Miller \cite{FokinaKalimullinMiller10} showed that there is a structure $\mc{A}$ with strong degree of categoricity $\mathbf{0}^{(\omega)}$. We note the well-known fact that one can also have $\mc{A}$ be a prime model. Our proof follows that of \cite{CsimaFranklinShore13}.

\begin{theorem}\label{thm:prime-0omega}
There is a computable structure $\mc{A}$ with strong degree of categoricity $\mathbf{0}^{(\omega)}$ such that $\mc{A}$ is a prime model of its theory.
\end{theorem}
\begin{proof}[Proof sketch]
The structure is just the disjoint union of infinitely many copies of each $\mc{E}_n$ for $n < \omega$. Theorem 3.1 of \cite{CsimaFranklinShore13} shows that this has strong degree of categoricity $\mathbf{0}^{(\omega)}$, and it is not hard to see using Lemma \ref{lem:moh} that this structure is prime.
\end{proof}

\section{C.E. In And Above a Limit Ordinal}

We begin this section by a short discussion of how we code a c.e.\ set into a structure. Consider a c.e.\ set $C$. If one knows, for each $n$, at what point the approximation to $C(n)$ has settled, then one can compute $C$. Moreover, one does not need to know exactly when $C$ settles, but just a point after which $C(n)$ has settled. In particular, any sufficiently large function can compute $C$. Moreover, $C$ itself can compute such a function. Following the terminology of Groszek and Slaman \cite{GroszekSlaman}, we say that $C$ has a self-modulus.

\begin{definition}[Groszek and Slaman \cite{GroszekSlaman}]
Let $F \colon \omega \to \omega$ and $X \subseteq \omega$. Then:
\begin{itemize}
	\item $F$ is a modulus (of computation) for $X$ if every $G \colon \omega \to \omega$ that dominates $F$ pointwise computes $X$.
	\item $X$ has a self-modulus if $X$ computes a modulus for itself.
\end{itemize}
\end{definition}

\noindent The self-modulus of a c.e.\ set $C$ is the function $f(n) = \mu s (C_s(n) = C(n))$. Groszek and Slaman showed that every $\Delta^0_2$ or $\alpha$-CEA set has a self-modulus. In fact, the self-modulus of a c.e.\ set has a nice form; it has a non-decreasing computable approximation.

\begin{definition}
A function $F \colon \omega \to \omega$ is limitwise monotonic if there is a computable
approximation function $f \colon \omega \times \omega \to \omega$ such that, for all $n$,
\begin{itemize}
	\item $F(n) = \lim_{s \to \infty} f(n,s)$.
	\item For all $s$, $f(n, s) \leq f(n, s + 1)$.
\end{itemize}
\end{definition}

\noindent In fact, it is well-known and an easy exercise to show that the sets of c.e.\ degree are exactly those with limitwise monotonic self-moduli. These remarks also relativize.

The next lemma encodes a limitwise monotonic function into the isomorphisms of copies of a computable structure. Any isomorphism dominates the limitwise monotonic function; but it does not seem to be the case that dominating the limitwise monotonic function is sufficient to compute isomorphisms.

\begin{lemma}\label{lem:main-constr}
Let $\alpha$ be a computable limit ordinal. Let $f \colon \omega \to \omega$ be limitwise monotonic relative to $\mathbf{0}^{(\alpha)}$.
There is a structure with computable copies $\mc{M}$ and $\mc{N}$ such that:
\begin{enumerate}
	\item Every isomorphism between $\mc{M}$ and $\mc{N}$ computes a function which dominates $f$.
	\item $f \oplus \mathbf{0}^{(\alpha)}$ computes an isomorphism between any two computable copies of $\mc{M}$ and $\mc{N}$.
\end{enumerate}
\end{lemma}
\begin{proof}
Let $\Phi$ be a computable operator such that $f(n) = \lim_{s \to \infty} \Phi^{\emptyset^{(\alpha)}}(n,s)$ and this is monotonic in $s$. Write $\emptyset^{(\alpha)} = \bigoplus_{\gamma < \alpha} \emptyset^{(\gamma)}$ for successor ordinals $\gamma < \alpha$. By convention, for $\beta < \alpha$, we say that $\Phi^{\emptyset^{(\beta)}}(n,s)$ converges if the computation $\Phi^{\emptyset^{(\alpha)}}(n,s)$ halts, but the only part of the oracle $\emptyset^{(\alpha)} = \bigoplus_{\gamma < \alpha} \emptyset^{(\gamma)}$ that is read during the computation is that part with $\gamma \leq \beta$. So if $\Phi^{\emptyset^{(\beta)}}(n,s) = m$ then $\Phi^{\emptyset^{(\alpha)}}(n,s) = m$, and because $\alpha$ is a limit ordinal, if $\Phi^{\emptyset^{(\alpha)}}(n,s) = m$ then $\Phi^{\emptyset^{(\beta)}}(n,s) = m$ for some successor ordinal $\beta < \alpha$.

Let $(\mc{A}_\beta)_{\beta < \alpha}$ and $(\mc{E}_\beta)_{\beta < \alpha}$ be as in \cref{lem:michael}.
We will construct the structures $\mc{M}$ and $\mc{N}$. They are the disjoint union of infinitely many structures $\mc{M}_n$ and $\mc{N}_n$, with $\mc{M}_n$ and $\mc{N}_n$ picked out by unary relations $R_n$.
The $n$th sort will code the value of $f(n)$.

Fix $n$. $\mc{M}_n$ will have infinitely many elements $(a_i)_{i \in \omega}$ satisfying a unary relation $S$. Each of these elements will be attached to, for each successor ordinal $\beta < \alpha$, a ``box'' $\mc{M}_{i,\beta}$ which contains within it a copy of either $\mc{A}_\beta$ or $\mc{E}_\beta$; each of the boxes are disjoint. By this we mean that there are binary relations $T_\beta$ such that $T_\beta(a_i,x)$ holds for exactly those $x \in \mc{M}_{i,\beta}$. $\mc{M}_{i,\beta}$ will be a structure in the language of \cref{lem:michael} and will be defined so that:
\begin{enumerate}
	\item $\mc{M}_{0,\beta} \cong \mc{A}_\beta$ for all $\beta$.
	\item $\mc{M}_{i,\beta} \cong \mc{E}_\beta$, $i \geq 1$, if there is $s$ such that $\Phi^{\emptyset^{(\beta)}}(n,s) \geq i$.
	\item $\mc{M}_{i,\beta} \cong \mc{A}_\beta$, $i \geq 1$, otherwise.
\end{enumerate}
Note that the condition in (2) is $\Sigma^0_\beta$ and so we can build such a structure $\mc{M}_n$ computably.

Similarly, $\mc{N}_n$ will have infinitely many elements $(b_i)_{i \in \omega}$, each of which is attached to, for each $\beta < \alpha$, a box $\mc{N}_{i,\beta}$ which contains within it:
\begin{enumerate}
	\item $\mc{N}_{i,\beta} \cong \mc{E}_\beta$ if there is $s$ such that $\Phi^{\emptyset^{(\beta)}}(n,s) > i$.
	\item $\mc{N}_{i,\beta} \cong \mc{A}_\beta$ otherwise.
\end{enumerate}
Again, the condition in (1) is $\Sigma^0_\beta$ and so we can build such a structure $\mc{N}_n$ computably.

\begin{claim} \label{claim:structure} Fix $n$.
\begin{enumerate}
	\item For each $j < f(n)$, there is $\beta < \alpha$ such that:
		\begin{itemize}
			\item for $\gamma < \beta$, $\mc{M}_{j+1,\gamma} \cong \mc{N}_{j,\gamma} \cong \mc{A}_\gamma$,
			\item for $\gamma \geq \beta$, $\mc{M}_{j+1,\gamma} \cong \mc{N}_{j,\gamma} \cong \mc{E}_\gamma$,
		\end{itemize}
	\item For each $j \geq f(n)$ and $\beta < \alpha$, $\mc{M}_{j+1,\beta} \cong \mc{N}_{j,\beta} \cong \mc{M}_{0,\beta} \cong \mc{A}_\beta$.
\end{enumerate}
\end{claim}
\begin{proof}
For (1), it is clear from the definitions of $\mc{M}_{j+1,\beta}$ and $\mc{N}_{j,\beta}$ that for all $\beta < \alpha$, $\mc{M}_{j+1,\beta} \cong \mc{N}_{j,\beta}$. Since $j < f(n)$, there is $s$ such that $\Phi^{\emptyset^{(\alpha)}}(n,s) = f(n) > j$. In particular, there must be some $\beta < \alpha$ such that there is $s$ with $\Phi^{\emptyset^{(\beta)}}(n,s) > j$. Let $\beta$ be the least such ordinal. Then for all $\gamma \geq \beta$, there is $s$ such that $\Phi^{\emptyset^{(\beta)}}(n,s) > j$, and so $\mc{M}_{j+1,\gamma} \cong \mc{N}_{j,\gamma} \cong \mc{E}_\gamma$. By choice of $\beta$, for $\gamma < \beta$, there is no $s$ such that $\Phi^{\emptyset^{(\beta)}}(n,s) > j$, and so $\mc{M}_{j+1,\gamma} \cong \mc{N}_{j,\gamma} \cong \mc{A}_\gamma$.

For (2), it is clear that $\mc{M}_{j+1,\beta} \cong \mc{N}_{j,\beta}$ for each $j \geq f(n)$ and each $\beta < \alpha$, and it is also clear that $\mc{M}_{0,\beta} \cong \mc{A}_\beta$ for each $\beta < \alpha$. If $j \geq f(n)$, then since $\Phi$ is limitwise monotonic approximation to $f$, $\Phi^{\emptyset^{(\beta)}}(n,s) \leq f(n) \leq j$ for all $s$ and $\beta$. Thus $\mc{N}_{j,\beta} \cong \mc{A}_\beta$ for all $\beta$.
\end{proof}

\begin{claim}
$\mc{M}$ and $\mc{N}$ are isomorphic.
\end{claim}
\begin{proof}
It suffices to show that for each $n$, $\mc{M}_n$ and $\mc{N}_n$ are isomorphic. Fix $n$. Using \cref{claim:structure}, we see that the map
\begin{align*}
a_0 &\mapsto b_{f(n)} & \\
a_{i} &\mapsto b_{i-1} && \text{when $0 < i \leq f(n)$} \\
a_i &\mapsto b_i && \text{when $i > f(n)$} \\
\end{align*}
extends to an isomorphism between $\mc{M}_n$ and $\mc{N}_n$.
\end{proof}

\begin{claim}
Any isomorphism between $\mc{M}$ and $\mc{N}$ can compute a function which dominates $f$.
\end{claim}
\begin{proof}
Let $g$ be an isomorphism between $\mc{M}$ and $\mc{N}$. We will compute, using $g$, a function $\hat{g}$ which dominates $f$. For each $n$, define $\hat{g}(n)$ as follows. Let $(a_i)_{i \in \omega}$ and $(b_i)_{i \in \omega}$ be the elements in the definition of $\mc{M}_n$ and $\mc{N}_n$. Then $\hat{g}(n)$ is the number satisfying $g(a_0) = b_{\hat{g}(n)}$.

To see that $\hat{g}(n) \geq f(n)$, we use \cref{claim:structure}. For each $\beta < \alpha$, $\mc{M}_{0,\beta} \cong \mc{A}_\beta$, but if $j < f(n)$, there is $\beta < \alpha$ such that $\mc{N}_{j,\beta} \cong \mc{E}_\beta$. Thus no isomorphism can map $a_0$ to $b_j$ for $j < f(n)$, and so $\hat{g}(n) \geq f(n)$.
\end{proof}

\begin{claim}
Given a computable copy $\widetilde{\mc{N}}$ of $\mc{N}$, $f \oplus \mathbf{0}^{(\alpha)}$ can compute an isomorphism between $\mc{N}$ and $\widetilde{\mc{N}}$.
\end{claim}

It is more convenient for the proof to consider $\mc{N}$ rather than $\mc{M}$ in this claim, but as they are isomorphic it does not matter which we choose.

\begin{proof}
For each $n$, let $\widetilde{\mc{N}}_n$ be the structure with domain $R_n$ in $\widetilde{\mc{N}}$. It suffices to compute an isomorphism $g$ between $\mc{N}_n$ and $\widetilde{\mc{N}}_n$ for each $n$. Inside of $\widetilde{\mc{N}}_n$, let $(c_i)_{i \in \omega}$ list the elements $x$ satisfying $S(x)$. For each $c_i$, let $\widetilde{\mc{N}}_{i,\beta}$ be the tree whose domain consists of the elements $y$ satisfying $T_\beta(c_i,y)$. To begin, we will define $g$ on $(b_i)_{i \in \omega} \subseteq \mc{N}_n$. Compute $f(n)$. Using $\mathbf{0}^{(\alpha)}$, look for $f(n)$ elements $c_i$ such that, for some $\beta < \alpha$, $\widetilde{\mc{N}}_{i,\beta} \cong \mc{E}_\beta$. This search is computable relative to $\mathbf{0}^{(\alpha)}$ by \cref{lem:michael} (2), and by \cref{claim:structure} we know that there are exactly $f(n)$ such elements and so the search will terminate after finding every such element. Rearranging $(c_i)_{i \in \omega}$, we may assume that these elements are $c_0,\ldots,c_{f(n) - 1}$.

Now, for each $k < f(n)$, find the least $\beta_k$ such that $\mc{N}_{k,\beta_k} \cong \mc{E}_{\beta_k}$, and the least $\gamma_k$ such that $\widetilde{\mc{N}}_{k,\gamma_k} \cong \mc{E}_{\gamma_k}$.  Again, this is computable in $\mathbf{0}^{(\alpha)}$ by \cref{lem:michael} (2). Note that we must ask $\mathbf{0}^{(\alpha)}$ to determine what $\beta_k$ and $\gamma_k$ are least. The sets $\{\beta_0,\ldots,\beta_{f(n)-1}\}$ and $\{\gamma_0,\ldots,\gamma_{f(n)-1}\}$ must be identical including multiplicity (but possibly in a different order) as $\widetilde{\mc{N}}_n$ and $\mc{N}_n$ are isomorphic. So by rearranging $(c_i)_{i \in \omega}$ once again we may assume that $\beta_k = \gamma_k$ for each $k < f(n)$.

We have now rearranged the list $(c_i)_{i \in \omega}$ so that for each $i$ and $\beta < \alpha$, $\mc{N}_{i,\beta} \cong \widetilde{\mc{N}}_{i,\beta}$. Define $g$ so that $g(a_i) = c_i$. For each $i$ and $\beta < \alpha$, $\mc{N}_{i,\beta} \cong \widetilde{\mc{N}}_{i,\beta}$ are isomorphic to either $\mc{A}_\beta$ or $\mc{E}_\beta$, which are uniformly $\mathbf{0}^{(\beta)}$-categorical (\cref{lem:michael} (3)), and we can compute using $\mathbf{0}^{(\alpha)}$ which case we are in. So we can define $g$ on $\mc{N}_{i,\beta}$ to be an isomorphism to $\widetilde{\mc{N}}_{i,\beta}$. Thus $g$ is an isomorphism from $\mc{N}_n$ to $\widetilde{\mc{N}}_n$.
\end{proof}

These claims complete the proof of the theorem.
\end{proof}

Using this lemma, and taking the limitwise monotonic function to be the self-modulus of a c.e.\ set, it is not hard to prove our main theorem.

{
\renewcommand{\thetheorem}{\ref{thm:cea0alpha-degreeofcat}}
\begin{theorem}
Let $\alpha$ be a computable limit ordinal and $\mathbf{d}$ a degree c.e.\ in and above $\mathbf{0}^{(\alpha)}$.
There is a computable structure with strong degree of categoricity $\mathbf{d}$.
\end{theorem}
\addtocounter{theorem}{-1}
}
\begin{proof}
Fix $\alpha$ and let $D \in \mathbf{d}$ be a set c.e.\ in and above $\mathbf{0}^{(\alpha)}$. Since $D$ is c.e.\ in and above $\mathbf{0}^{(\alpha)}$, it has a self-modulus $f$ that is limitwise monotonic relative to $\mathbf{0}^{(\alpha)}$. Consider the structure $\mc{M}$ constructed in \cref{lem:main-constr} for this $f$. We will enrich this structure slightly to produce a new structure $\mc{S}$. Let $\mc{S}_\alpha$ be the computable structure with strong degree of categoricity $\mathbf{0}^{(\alpha)}$ constructed in Theorem 3.1 of Csima, Franklin and Shore \cite{CsimaFranklinShore13}. The new structure $\mc{S}$ consists of $\mc{M}$ and a disjoint copy of $\mc{S}_\alpha$, and a new unary relation $R$ such that $R(x)$ holds exactly when $x$ belongs to the copy of $\mc{S}_\alpha$. We claim that $\mc{S}$ has strong degree of categoricity $\mathbf{d}$.

 First, suppose that $\mc{T}$ is some other computable copy of $\mc{S}$. We will show that there is a $\mathbf{d}$-computable isomorphism between $\mc{S}$ and $\mc{T}$. Using the relation $R$, we may identify the component of $\mc{T}$ isomorphic to $\mc{S}_\alpha$. Since $\mc{S}_a$ has (strong) degree of categoricity $\mathbf{0}^{(\alpha)} \leq \mathbf{d}$, we can $\mathbf{d}$-computably find an isomorphism between the copies of $\mc{S}_\alpha$ in $\mc{S}$ and $\mc{T}$. We can also identify the component isomorphic to $\mc{M}$ in each structure. By choice of $\mc{M}$, any two such copies have an isomorphism between them computable in $f \oplus \mathbf{0}^{(\alpha)}$, and $D$ can compute this self-modulus $f$. Hence $\mathbf{d}$ can computably produce such an isomorphism, since it can compute $f \oplus \mathbf{0}^{(\alpha)}$. Gluing these two isomorphisms together gives us the result.

Since $\mc{S}_\alpha$ has strong degree of categoricity $\mathbf{0}^{(\alpha)}$, there is a computable copy $\hat{\mc{S}}_\alpha$ of $\mc{S}_\alpha$ such that every isomorphism between the two computes $\mathbf{0}^{(\alpha)}$. Let $\widetilde{\mc{S}}$ be a computable copy of $\mc{S}$ built in the following way.  Rather than using the ``standard'' copy $\mc{S}_\alpha$, use the ``hard'' copy $\hat{\mc{S}}_\alpha$ of $\mc{S}_\alpha$. Additionally, rather than using $\mc{M}$, instead use $\mc{N}$ as built in \cref{lem:main-constr}. Any isomorphism between $\mc{S}_\alpha$ and $\hat{\mc{S}}_\alpha$ computes $\mathbf{0}^{(\alpha)}$, and any isomorphism between $\mc{M}$ and $\mc{N}$ must compute a function that dominates $f$. Let $g$ be any isomorphism between $\mc{S}$ and $\widetilde{\mc{S}}$. Then by using $R$, we can restrict $g$ to an isomorphism between $\mc{S}_\alpha$ and $\hat{\mc{S}}_\alpha$ and hence $g$ can compute $\mathbf{0}^{(\alpha)}$. Since $g$ can also be restricted to an isomorphism between $\mc{M}$ and $\mc{N}$, it must compute a function dominating $f$. But $f$ is a modulus for $D$ computable in $\mathbf{0}^{(\alpha)}$, and hence $g$ must be able to compute $D$ since it can compute $\mathbf{0}^{(\alpha)}$ and a function dominating $f$. Hence $g$ can compute $\mathbf{d}$.
\end{proof}

We now turn to prime models, working above $\mathbf{0}^{(\omega)}$. Essentially, our work here is to check that in taking $\alpha = \omega$ in the previous theorem and lemma, the construction results in a prime model.

\begin{lemma}\label{lem:main-constr-prime}
Let $f \colon \omega \to \omega$ be limitwise monotonic relative to $\mathbf{0}^{(\omega)}$.
There is a prime model with two computable copies $\mc{M}$ and $\mc{N}$ such that:
\begin{enumerate}
	\item Every isomorphism between $\mc{M}$ and $\mc{N}$ computes a function which dominates $f$.
	\item $f \oplus \mathbf{0}^{(\omega)}$ computes an isomorphism between any two computable copies of $\mc{M}$ and $\mc{N}$.
\end{enumerate}
\end{lemma}
\begin{proof}
The construction is exactly the same as that of \cref{lem:main-constr} with $\alpha = \omega$. We refer to the structures $\mc{A}_\beta$ and $\mc{E}_\beta$ of \cref{lem:michael} as $\mc{A}_n$ and $\mc{E}_n$, $n < \omega$, but of course these are the same. It remains to argue, using the properties from \cref{lem:moh} which hold only for the structures $\mc{A}_n$ and $\mc{E}_n$ with $n$ finite, that the resulting structure $\mc{N}$ is prime.

Recall that $\mc{N}$ is the disjoint union of structures $\mc{N}_n$, each of which satisfies the relation $R_n$. So it suffices to show that the structures $\mc{N}_n$ are prime. $\mc{N}_n$ was defined as follows: there were infinitely many elements $(b_i)_{i \in \omega}$ (satisfying the unary relation $S$), each of which is attached to (by binary relations $T_m$), for each $m < \omega$, a box $\mc{N}_{i,m}$ which contains within it:
\begin{enumerate}
	\item $\mc{N}_{i,m} \cong \mc{E}_m$ if there is $s$ such that $\Phi^{\emptyset^{(m)}}(n,s) > i$.
	\item $\mc{N}_{i,m} \cong \mc{A}_m$ otherwise.
\end{enumerate}
By \cref{claim:structure} of \cref{lem:main-constr}, for each $i$, either $i < f(n)$ and there is some $m_i < \omega$ such that:
\begin{itemize}
			\item for $\ell < m_i$, $\mc{N}_{i,\ell} \cong \mc{A}_\ell$,
			\item for $\ell \geq m_i$, $\mc{N}_{i,\ell} \cong \mc{E}_\ell$,
\end{itemize}
or $i \geq f(n)$ and for all $m < \omega$, $\mc{N}_{i,m} \cong \mc{A}_m$. Note that the sequence $\{m_i\}_{i < f(n)}$ is non-decreasing.

By \cref{lem:moh} (2), for $i < f(n)$, the automorphism orbit of $b_i$ is determined by the first-order formula with free variable $x$ which expresses that $S$ holds of $x$, that the structure with domain $T_{m_i}(x,\cdot)$ satisfies $\varphi_{m_i}$ (and so is isomorphic to $\mc{E}_{m_i}$), and that the structure with domain $T_{m_{i}-1}(x,\cdot)$ satisfies $\neg \varphi_{m_{i}-1}$ (and so is isomorphic to $\mc{A}_{m_{i}-1}$). For $i \geq f(n)$, the automorphism orbit of $b_i$ is determined by the first-order sentence with free variable $x$ which expresses that $S$ holds of $x$, and that the structure with domain $T_{m_{f(n)-1}}(x,\cdot)$ satisfies $\neg \varphi_{m_{f(n)-1}}$ (and so is isomorphic to $\mc{A}_{m_{{f(n)-1}}}$).

Fix a tuple $\bar{c}$ from $\mc{N}_n$. We will give a first-order formula defining the orbit of $\bar{c}$. We may assume that whenever $\bar{c}$ contains an element of $\mc{N}_{i,m}$, $\bar{c}$ contains $b_i$ as well. We can break the tuple $\bar{c}$ up into finitely many elements $b_{i_1},\ldots,b_{i_k}$ and finitely many tuples $\bar{c}_{i,m}$ from $\mc{N}_{i,m}$. The orbit of $\bar{c}$ is determined by the orbits of $b_{i_1},\ldots,b_{i_k}$ (each of which is determined by a first-order formula as described in the previous paragraph), the fact that $T_m(b_{i},y)$ holds for any $y \in \bar{c}_{i,m}$, and the orbits of each of the tuples $\bar{c}_{i,m}$ within $\mc{N}_{i,m}$. The latter orbits are first-order definable by \cref{lem:moh} (3).
\end{proof}

{
\renewcommand{\thetheorem}{\ref{thm:main-prime}}
\begin{theorem}
Let $\mathbf{d}$ be a degree c.e.\ in and above $\mathbf{0}^{(\omega)}$.
There is a computable prime model $\mc{A}$ with strong degree of categoricity $\mathbf{d}$.
\end{theorem}
}
\begin{proof}
 The construction of such a model is similar to \cref{thm:cea0alpha-degreeofcat}, except we replace $\mc{M}$ and $\mc{N}$ from \cref{lem:main-constr} with those $\mc{M}$ and $\mc{N}$ from \cref{lem:main-constr-prime} (which are actually the same structures, if $\alpha = \omega$), and we also replace the ``easy'' and ``hard'' copies of $\mc{S}_\alpha$ with copies of the structure from \cref{thm:prime-0omega} such that any isomorphism between them computes $\mathbf{0}^{(\omega)}$. The same argument from \cref{thm:cea0alpha-degreeofcat} shows that this new structure has strong degree of categoricity $\mathbf{d}$. It remains to show that such models are prime; they are the disjoint union of prime structures, distinguishable by the relation $R$, and hence must be prime themselves.
\end{proof}

\bibliography{References}
\bibliographystyle{alpha}

\end{document}